\documentclass{amsart}
\newtheorem{theorem}{Theorem}[section]
\newtheorem{lemma}[theorem]{Lemma}
\newtheorem{corollary}[theorem]{Corollary}
\newtheorem{proposition}[theorem]{Proposition}

\theoremstyle{definition}
\newtheorem{definition}[theorem]{Definition}

\newtheorem{assumption}[theorem]{Assumption}

\theoremstyle{remark}
\newtheorem{remark}[theorem]{Remark}
\numberwithin{equation}{section}
\usepackage{color,graphicx,amssymb}
\usepackage{graphicx}
\allowdisplaybreaks
\begin{document}
\title{Decay of correlation for random intermittent maps}
\author{Wael Bahsoun$^{\dagger}$}
    %Address of record for the research reported here
\address{Department of Mathematical Sciences, Loughborough University,
Loughborough, Leicestershire, LE11 3TU, UK}
\email{$\dagger$ W.Bahsoun@lboro.ac.uk, $\ddagger$ Y.Duan@lboro.ac.uk}
\author{Christopher Bose$^{*}$}
    %Address of record for the research reported here
\address{Department of Mathematics and Statistics, University of Victoria,
   PO BOX 3045 STN CSC, Victoria, B.C., V8W 3R4, Canada}
\email{$*$ cbose@uvic.ca}
\author{Yuejiao Duan$^{\ddagger}$}
\date{\today}
%\dedicatory{This paper is dedicated to our authors.}
\begin{abstract}
We study a class of random transformations built over finitely many intermittent maps sharing a {\em common} indifferent fixed point. Using a Young-tower technique, we show that the map with the fastest relaxation rate dominates the asymptotics. In particular, we prove that the rate of correlation decay for the annealed dynamics of the random map is the same as the {\em sharp rate} of correlation decay for the map with the fastest relaxation rate.
\end{abstract}
\maketitle
\pagestyle{myheadings}
\markboth{Decay of correlation for random intermittent maps}{W. Bahsoun, C. Bose, Y. Duan}

\section{introduction}
General statistical properties of deterministic expanding maps of the interval with a neutral fixed point are by now well understood. In \cite{Pi} Pianigiani proved existence of invariant densities of such maps. In \cite{H,LSV,Y} it was independently proved that such maps exhibit a polynomial rate of correlation decay.  Later Gou\"ezel \cite{G1} showed the rate obtained in \cite{Y} is in fact sharp. The slow mixing behaviour of such maps made them a useful testing ground for physical problems with intermittent behaviour: systems whose orbits spend very long time in a certain small part of the phase space.

\bigskip

In this paper we are interested in studying i.i.d.\ randomized compositions of two intermittent maps sharing a common indifferent fixed point. It is intuitively clear that the annealed\footnote{\emph{Annealed
dynamics} refers to the randomized dynamics, averaged over the randomizing space, see Subsection \ref{rep} and Theorem \ref{main}.  This should be contrasted with the notion of \emph{quenched dynamics}, the behaviour of the system with one random choice of the randomizing sequence.  The term \emph{almost sure dynamics} is also used to refer to quenched dynamics.} dynamics of the random process will also have a polynomial rate of correlation decay. However, we are interested in the following question: How do the asymptotics of the random map relate to those of the original maps; in particular, the rate of correlation decay? 

\bigskip

We show that the map with the fast relaxation rate dominates the asymptotics (see Theorem \ref{main} for a precise statement). Interestingly, in our setting, the map with slow relaxation rate is allowed to be of `boundary-type', and consequently admit an infinite ($\sigma$-finite) invariant measure, but the random system will always admit an absolutely continuous invariant {\em probability} measure. We obtain our result by using a version of the skew product representation\footnote{The skew product studied in this paper is also related to a skew product studied by Gou\"ezel in \cite{G2}. See Remark \ref{Gouz-comp} for more details.} studied in \cite{BBQ} and a Young-tower technique \cite{Y}.  
 
 \bigskip
 
In Section 2 we introduce our random system and its skew product representation. The statement of our main result Theorem \ref{main} is also in Section 2. In Section 3 we build a Young-tower for the skew product representation. Proofs, including the proof of Theorem \ref{main}, are in Section 4.
\section{Setup and Statement of the main result}
\subsection{A random dynamical system}
Let $(I,\mathfrak B(I),m)$ be the measure space, with $I=[0,1], \mathfrak{B}(I)$ the Borel $\sigma-$algebra and $m$ being Lebesgue measure. By a \emph{Liverani-Saussol-Vaienti \textnormal{(LSV)}-map} we mean a member of the parameterized family of maps on $I$ given by
\begin{equation}\label{eqn_lsv}
T_{\alpha}(x)=\begin{cases}
       x(1+2^{\alpha}x^{\alpha}) \quad x\in[0,\frac{1}{2}]\\
       2x-1 \quad \quad \quad x\in(\frac{1}{2},1]
       \end{cases}.
\end{equation}
Here the parameter $\alpha\in (0,\infty)$.  Each LSV map has a neutral fixed point at $x=0$.  For  $0<\alpha < 1$,  $T_\alpha$ admits a finite, absolutely continuous invariant measure while for $\alpha \geq 1$ the absolutely continuous invariant measure is $\sigma-$ finite. See \cite{Pi} and \cite{T} for the some of the earliest results of this type.

Let $0<\alpha_1 < \alpha_2 < \dots < \alpha_r \leq 1$. We consider a random map  $T$ which is given by:
\begin{equation}\label{eqn_randommap}
T(x):=\{T_{\alpha_1}(x),T_{\alpha_2}(x), \dots T_{\alpha_r}(x); p_{1}, p_{2}, \dots p_r \},
\end{equation} 
where 
$p_i>0$ and $\sum p_i = 1$.
Note that all the individual maps share a single common neutral fixed point at $x=0$.

\begin{assumption}
Since nothing we will do in the sequel depends on $r$, the number of maps making up the random map, we will restrict the discussion to the case $r=2$ and denote the parameters 
$$0<\alpha<\beta\leq 1.$$
At the same time, this will simplify our notation:
\begin{equation}\label{eqn_randommap2}
T(x):=\{T_\alpha(x), T_\beta(x); p_1, p_2\}.
\end{equation}
\end{assumption}

The random map $T$  in (\ref{eqn_randommap2}) maybe viewed as a Markov process with transition function 
$$\mathbb P(x,A)=p_{1}{\bf 1}_{A}(T_{\alpha}(x))+p_{2}{\bf 1}_{A}(T_{\beta}(x))$$ 
of a point $x\in I$ into a set $A\in \mathfrak B(I)$. The transition function induces an operator, $E_T$, acting on measures; i.e., if $\mu$ is a measure on $(I, \mathfrak B)$,
$$(E_T\mu)(A)=p_1\mu(T_{\alpha}^{-1}(A)) + p_2\mu(T_{\beta}^{-1}(A)).$$
A measure $\mu$ is said to be $T$-invariant if $$\mu=E_T\mu,$$
and $\mu$ is said to be an absolutely continuous invariant measure if $d\mu=f^*dm$, $\int_I f^*dm=1$.
To study absolutely continuous invariant measures, we introduce the transfer operator (Perron-Frobenius) of
the random map $T$:
$$(P_{T}f)(x)=p_1P_{T_{\alpha}}\left(f\right)(x)+p_2P_{T_{\beta}}\left(f\right)(x),$$
where $P_{T_{\alpha}}, P_{T_{\beta}}$ denote the transfer operators associated with the $T_{\alpha}, T_{\beta}$ respectively. Then it is a straight-forward computation to show that a measure $\mu=f^*\cdot m$ is absolutely continuous invariant measure if
$$P_Tf^*=f^*.$$
\subsection{Skew product representation}\label{rep}
By the {\em annealed} dynamics of the random map we mean the statistics of the random dynamical system averaged over the randomizing space (see \cite{ALS} for a general treatment of annealed versus quenched interpretation). Probabilistic aspects of $T$, in particular the correlation decay of the annealed dynamics, are frequently studied by constructing a Young-tower for the skew-product representation of $T$. For this purpose, we are going to use a version\footnote{The results obtained in \cite{BBQ} are valid for any class of measurable non-singular maps on $\mathbb R^q$, without any regularity assumptions. Moreover in \cite{BBQ}, the probability distribution on the noise space is allowed to be place-dependent.} of the skew product representation which was studied in \cite{BBQ}. Define the skew product transformation $S(x,\omega):I\times I\to I\times I$ by 
\begin{equation}\label{skew}
S(x,\omega)=(T_{\alpha(\omega)},\varphi(\omega)),
\end{equation}
where
\begin{equation}\label{alpha}
\alpha(\omega)=\begin{cases}
       \alpha \quad ,\omega\in[0,p_1)\\
       \beta \quad ,\omega\in[p_1,1]
       \end{cases}
;\quad\quad
\varphi(\omega)=\begin{cases}
       \frac{\omega}{p_1} \quad \quad ,\omega\in[0,p_1)\\
       \frac{\omega-p_1}{p_2} \quad ,\omega\in[p_1,1]
       \end{cases}.
\end{equation}       
We denote the transfer operator associated with $S$ by $\mathcal L_S$: for $g\in L^1(I\times I)$ and measurable $A \subseteq I \times I$,
$$
  \int_{S^{-1}A} \, g\, d(m \times m)(x,\omega) = \int_A \, \mathcal L_S g \,
  d(m \times m)(x,\omega).
$$
Then a measure $\nu$, such that $d\nu=g^*d(m\times m)$ and $\int_{I\times I}g^*d(m\times m)=1$, is an absolutely continuous $S$-invariant measure if 
$$\mathcal L_S g^*=g^*.$$
In \cite{BBQ}, Theorem 5.2 it is shown that if  $g \in L^1(I\times I)$ and $\mathcal L_S g =\lambda g$ with $|\lambda|=1$, then 
$$g(x,\omega)=f(x)\cdot {\bf 1}(\omega)$$
and $P_T f = \lambda f$, 
 that is, $g$ depends only on the spatial coordinate $x$ and 
 as a function of $x$ only, is also an eigenfunction for $P_T$. Setting $\lambda =1$ we obtain
$\mathcal L_S g^*=g^*$ if and only if  $g^*(x, \omega) = f^*(x)$ with  $P_T f^* = f^*$.  Consequently there is a one to one correspondence between invariant densities for $S$ and invariant densities for $T$.   Moreover, dynamical properties such as ergodicity, number of ergodic components or weak-mixing, properties that are determined by peripheral eigenfunctions,  can be determined via either system. 

On the other hand, properties like correlation decay (or even strong mixing) cannot be established by peripheral spectrum alone. 

\begin{definition}\label{def_correlation} 
Suppose $\tau: X \rightarrow X$ preserves the measure $\mu$ on $X$. 
For $f\in L^\infty(X,\mu)$ and $g \in L^1(X, \mu)$ denote by 
$$Cor_n(f,g)= Cor_{n, \tau}(f,g) := \int_X  f \circ \tau^n \cdot g \, d\mu 
-\int_X f\, d\mu \,  \int_X g \, d\mu.$$
\end{definition}

Normally we will simply write $Cor_n(f,g)$ when the map being applied is understood. 

Estimates on correlation decay are known in many dynamical settings. 
For example, it was shown in \cite{Y} that for $f  \in L^\infty$, 
$g$ H\"older continuous on $I$ and $\tau=T_\alpha,~0<\alpha<1$ an LSV-map,   $|Cor_{n, T_\alpha} (f,g)| 
= O(n^{1 - \frac{1}{\alpha}})$.  Gou\"ezel in \cite{G1} proved that this rate is sharp.

Our main result, Theorem \ref{main}, establishes exactly the same rate of correlation decay for the random map. 

%We can extend this notion to random maps $T$ on $X$ by duality\footnote{Another natural definition would be to pick a single randomizing sequence $\omega$ and omit the averaging over randomizers $\omega$ in the correlation integral. This is the so-called \emph{chilled} correlation whereas we are interested in the annealed correlation}:
%$$Cor_n(f, g) = Cor_{n, T}(f, g) := \int( \int_X f  \cdot P^n_{T(\omega)}g \, d\mu)d\omega
%- \int_X f \, d\mu \int_X g \, d\mu.$$
%Note that in the case of our random map $T$ defined in equation (\ref{eqn_randommap2}) and the skew $S$ in equation (\ref{skew}), we have
%\begin{equation}\label{eqn_equal_correlations}
%Cor_{n, S}(\hat f, \hat g) = Cor_{n, T}(f,g)
%\end{equation}
%when $\hat f(x, \omega) := f(x)$ and $\hat g(x, \omega) = g(x)$ and all measures are computed with respect to the invariant $g^*$ or $f^*$ as described above. 

\subsection{Statement of the main result}
\begin{theorem}\label{main} Let $0<\alpha<\beta\le 1$ and $S$ be as defined in Subsection \ref{rep}. Then
\begin{enumerate}
\item $S$ admits a unique absolutely continuous invariant probability measure $\nu$;
\item $(S, \nu)$ is mixing;
 \item for $\phi\in L^{\infty}(I\times I, m\times m)$ and $\psi$ a H\"older continuous function on $I\times I$  we have 
 $$|Cor_{n,S}(\phi, \psi)|=
 \mathcal{O}(n^{1-\frac{1}{\alpha}});$$
 %\item for $\phi\in L^{\infty}(I, m)$ and $\psi$ a H\"older continuous function on $I$ we have 
% $$|\int_{I\times I}\phi \cdot P_{T(\omega)}^n\psi \,d\nu d\omega- \int_I \phi dm \int_{I} \psi dm|=
% \mathcal{O}(n^{1-\frac{1}{\alpha}}).$$
 
\end{enumerate}
\end{theorem}
\begin{remark}
Our main goal in Theorem \ref{main} is not so much to show that $S$ has polynomial rate of correlation decay, but to discover how the correlation decay for $S$ relates to those of the original maps. Indeed, if $0<\alpha<\beta<1$, and without any further conditions on $\alpha$ and $\beta$, one can easily obtain, by just using the rough estimates contained in Lemma \ref{lem_rough_estimates} and the Young tower construction detailed in the next two sections, an upper bound on the rate of order $\mathcal{O}(n^{1-\frac{1}{\beta}})$; that is, the rate of decay is at least as fast as the slowest escape rate map.   What we have shown in Theorem \ref{main} is that the actual decay rate of the random map is \emph{completely determined} by the faster escape rate of the map $T_\alpha$.
\end{remark}
\begin{remark}
It is worth noting that in Theorem \ref{main}, $\beta\le 1$. The case when $\beta=1$ is interesting on its own since in this case the map $T_{\beta}$ admits only an infinite ($\sigma$-finite) absolutely continuous invariant measure, but Theorem \ref{main} shows that the skew product $S$, and hence the random map $T$, admits a unique absolutely continuous invariant probability measure.  
\end{remark}
\begin{remark}\label{Gouz-comp} Limit theorems for the following related skew product were studied by Gou\"ezel in \cite{G2}: 
$$S(x,\omega)=(T_{\alpha(\omega)},4\omega),$$
with $\omega \in \mathbb S ^1$ and
$T_\alpha(\omega)$ being a random choice of LSV-map from Equation 
\ref{eqn_lsv}. 
For the randomizing process, it is further assumed that 
\begin{enumerate}
\item $\alpha(\omega)$ is $C^2$;
\item $0<\alpha_{\text{min}}<\alpha_{\text{max}}<1$;
\item $\alpha(\omega)$ takes the value $\alpha_{\text{min}}$ at a unique point $\omega_0\in\mathbb S^1$, with $\alpha''(\omega_0)>0$;
\item $\alpha_{\text{max}}<\frac32\alpha_{\text{min}}$.
\end{enumerate} 
Under the above conditions, using a result of P\`ene \cite{P} (see \cite{G2} Theorem B.1), Gou\"ezel (\cite{G2}, Theorem 4.1) obtained asymptotics that would lead to a 
correlation decay rate of order $\mathcal{O}(\sqrt{\log n}\cdot n^{1-\frac{1}{\alpha_{\text{min}}}})$.

Gou\"ezel remarks in \cite{G2} that in his setting the conditions 
$\alpha_{\text{max}} < 1$ and $\alpha_{\text{max}} < \frac{3}{2} \alpha_{\text{min}}$ may be technical artifacts, arising from the method of proof.   
\end{remark}
\begin{remark}\label{new:re}
In our i.i.d.\ setting, we rely on relatively simple (and classical) estimates on large deviations for i.i.d.\ randomizers and show, for all $0<\alpha < \beta \leq 1$, that the correlation decay rate of the annealed dynamics is exactly the same as the correlation decay rate of the fastest mixing map $T_\alpha$. Note that the constraint $\beta \leq 1$ still may be a technical artifact, a consequence of our use of the negative Schwarzian derivative and Koebe Principle in Lemma \ref{distlog}, rather than some underlaying obstruction.  Indeed, the probabilistic analysis of our model (see Proposition \ref{prop_expectation_estimates}) is valid for the full range of parameter values $0<\alpha < \beta < \infty$, so it may be possible to extend our results to the case of $0< \alpha < 1 < \beta$ for the (presumably) finite invariant measure case or even $1\leq \alpha < \beta < \infty$, by following the techniques of \cite{MT}, where one expects the invariant measure to be only $\sigma$-finite. 
\end{remark}

\section{A Young-tower for $S$}
\subsection{Notation}
Set
$$T^n_\omega(x):=T_{\alpha(\varphi^{n-1}\omega)}\circ...\circ T_{\alpha(\varphi\omega)}\circ T_{\alpha(\omega)}(x).$$
Then
$$S^n(x,\omega)=(T^n_\omega(x),\varphi^{n}(\omega)).$$
Also, set
$$P^n_\omega:=p_{\alpha(\varphi^{n-1}\omega)}\times...\times p_{\alpha(\varphi\omega)}\times p_{\alpha(\omega)},$$
where $ p_{\alpha(\omega)}=p_1,$ for $\alpha(\omega)=\alpha$ and $ p_{\alpha(\omega)}=p_2,$ for $\alpha(\omega)=\beta.$ 
We define two sequences of random points $\{x_n(\omega)\}$ and $\{x'_n(\omega)\}$ in $[0,1]$ which will be useful in the construction of a suitable Young tower. The points $x_n(\omega)$ lie in $(0, 1/2]$. Set
\begin{equation}\label{def_backorbit}
x_1(\omega)\equiv\frac{1}{2}\text{ and } x_n(\omega)=T^{-1}_{\alpha(\omega)}\mid_{[0,\frac{1}{2}]}[x_{n-1}(\varphi\omega)], n\geq 2.
\end{equation}
Observe that with this notation, 
$$S(x_n(\omega),\omega) 
= (T_{\alpha(\omega)}(x_n(\omega)), \varphi \omega) = (x_{n-1}(\varphi \omega), \varphi \omega). $$
%Note that $x_{n-1}(\varphi\omega)$ and $x_{n-1}(\omega)$ are two different points with same $\omega.$ 
The points $\{x'_n(\omega)\}$  lie in $(\frac{1}{2},1]$, defined by
\begin{equation}\label{def_backorbit_prime}
x'_0(\omega)\equiv 1,x'_1(\omega)\equiv\frac{3}{4} \text{ and } x'_n(\omega)=\frac{x_n(\varphi \omega)+1}{2},n\geq 2,
\end{equation}
that is, $\{x'_n(\omega)\}$ are preimages of $\{x_n(\varphi \omega)\}$ in  $(\frac{1}{2},1]$ under the right branch $2x-1.$
\begin{figure}[b]  %  figure placement: here, top, bottom, or page
   \centering
   \includegraphics[width=3.5 in]{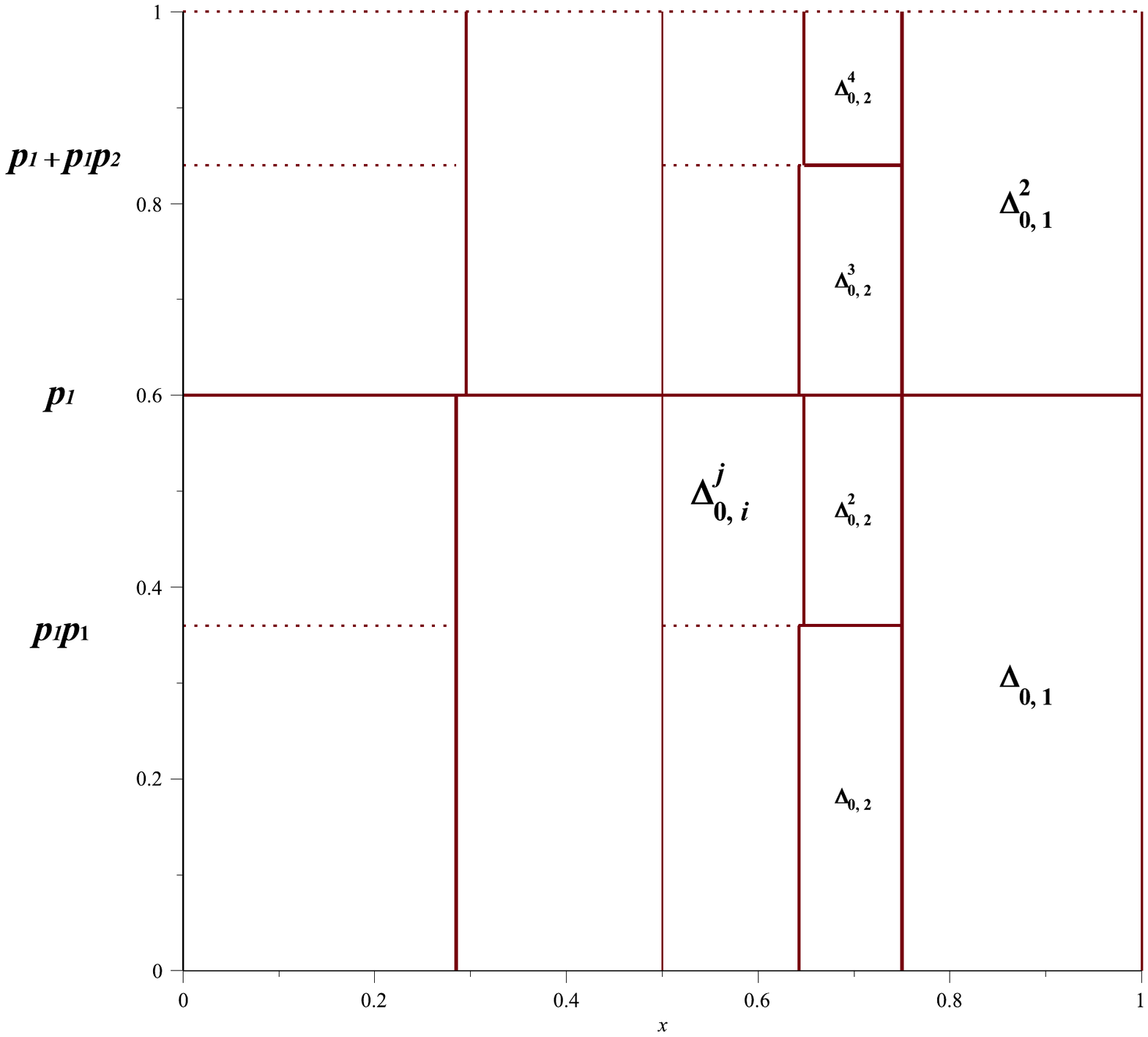}
   \caption{An example of the base of the tower when $\alpha=0.5$ $\beta=0.7$, $p_1=0.6$.}
   \label{fig:base_partition}
\end{figure} 
\subsection{A tower for $S$}
Let $\Delta_0=(\frac{1}{2},1]\times[0,1)$. Let $R: \Delta_0\to \mathbb{Z}^+$ be the first return time function and $S^R : \Delta_0\to \Delta_0$ be the return map. $\Delta_0$ is referred to as the base of the tower $\Delta$ which is given by
$$\Delta :=\{(z,n) :  z\in \Delta_0 \text{ and }  n=0,1,...,R(z)-1 \}.$$
Let $F:\Delta \to \Delta $ be the map acting on the tower as follows:
$$F(z,l)=\begin{cases}
       (z, l+1),  & \mbox{ if } l < R(z)-1,\\
       (S^R(z),0) ,& \mbox{ if } l = R(z)-1
       \end{cases}$$
We refer to $\Delta_{l}:= \Delta\cap\{n=l\}$ as the $l^\text{th}$ level of the tower. For $n\geq 1$, set $I_n(\omega):=(x_{n+1}(\omega),x_{n}(\omega)]$ and $J_n(\omega):=(x'_{n}(\omega),x'_{n-1}(\omega)]$.
%and observe that $$R(x,\omega)=n, \text{ for } (x,\omega)\in J_n(\omega)\times [0,1).$$
Observe that every point in $J_n(\omega)$ will return to $(\frac{1}{2},1]$ in $n$ steps 
under the random iteration $T^n_{\omega}$ as follows:
$$J_n(\omega)\rightarrow I_{n-1}(\varphi \omega)\rightarrow I_{n-2}(\varphi^2 \omega)\rightarrow...\rightarrow I_{1}(\varphi^{n-1}\omega)\rightarrow(\frac{1}{2},1].$$
Next we partition $\Delta_0$, into subsets $\Delta_{0,i}$, $i=1,2,\dots$ where
$$\Delta_{0,i}:= \{ (x,\omega) ~|~ x \in J_i(\omega) \}$$
and then further partition each $\Delta_{0,i}$ into subsets $\Delta^j_{0,i}, \, j=1, 2, \dots 2^i$ according to the $2^i$ possible
values  of the
string $\alpha(\omega), \alpha(\varphi \omega),  \dots \alpha(\varphi^{i-1}\omega)$.
Defined this way,  $S^i$ maps each subset $\Delta^j_{0,i}$ bijectively to $\Delta_0$. 

For example, in the case $i=2,$ there are four sets $\Delta^j_{0,2}$ on which 
$R=2$ and such that $S^R$ maps each set bijectively to $\Delta_0$:
$$\Delta^j_{0,2}=\begin{cases}
       J_2(\omega)\times [0,p^2_1)     & \mbox{ , if } j=1,\\
       J_2(\omega)\times [p^2_1,p_1)   & \mbox{ , if } j=2,\\
       J_2(\omega)\times [p_1,p_1+p_1\cdot p_2) & \mbox{ , if } j=3,\\
       J_2(\omega)\times [p_1+p_1\cdot p_2,1) & \mbox{ , if } j=4.\\
       \end{cases}$$
To summarize, 
$$\Delta_{0,i}=\bigcup\limits_{j=1}\limits^{2^i}\Delta^j_{0,i}.$$
and for the base,
$$\Delta_0 = \bigcup\limits_{i=1}\limits^{\infty}\bigcup\limits_{j=1}\limits^{2^i}\Delta^j_{0,i}$$
where, for every $i$ and  $j=1,2,...,2^i,$
$$R\mid_{\Delta^j_{0,i}}=i.$$
Finally the tower  $\Delta$ is partitioned by
$$\Delta=\bigcup\limits_{i=1}\limits^{\infty}\bigcup\limits_{l=0}\limits^{i-1}(\bigcup\limits_{j=1}\limits^{2^i}\Delta^j_{l,i}).$$
An example of the partition of the base of the tower and a few  images  of 
the partition elements under $S$  is presented in Figure \ref{fig:base_partition}.

\subsection{Using Young's technique to prove Theorem \ref{main}}
We say $s(z_1,z_2)$ is a \textit{separation time} for $z_1,z_2\in \Delta_0$ if $s$ is the smallest $n\geq 0$ such that $(F^R)^n(z_1)$ and  $(F^R)^n(z_2)$ lie in distinct $\Delta_{0,i}^j$. Also let $\hat R:\Delta\to\mathbb Z$ be the function defined by
$$\hat R(x,\omega)=\text{ the smallest integer } n\ge 0 \text{ s.t. } F^n(x,\omega)\in\Delta_0.$$
To prove Theorem \ref{main}, we have to:
\begin{itemize}\item[(A)] Prove that $\int_{I\times I}Rd(m\times m)<\infty$, and establish the asymptotic estimate $(m\times m)\{\hat R>n\}=\mathcal{O}(n^{1-\frac1\alpha})$,
\item[(B)] Establish the bounded distortion conditions on the return map: there exists $0<\theta<1$ and $C(F)>0$ such that
\begin{eqnarray*}\label{Jacob}
\left|\frac{DF^R(z_1)}{DF^R(z_2)}-1\right|\leq C(F)\cdot\theta^{s(F^R(z_1),F^R(z_2))}, \forall i=1,2,...,\forall j=1,...,2^i,\forall z_1,z_2\in \Delta^j_{0,i},
\end{eqnarray*}
\item[(C)] Confirm that the return times are aperiodic. 
\end{itemize}
(A) is established by Proposition \ref{prop_expectation_estimates} in Subsection \ref{A} while (B) is the content of  
Proposition \ref{distortion} in Subsection \ref{B}.  Since 
we have all possible integer return times, (C) is immediate. 
It is interesting to note
that the upper bound constraint $\beta \leq 1$ specified in our main result,  Theorem \ref{main}, is only 
used in Proposition \ref{distortion}, so the tower asymptotics 
detailed in (A) hold for all
pairs $0<\alpha < \beta< \infty$.

%%%%%%%%%%%%%% PROOFS %%%%%%%%%%%%

\section{Proofs}
%Throughout this section $C$ will denote various positive constants which are independent of $n$ and $\omega$.

\subsection{Estimates on the return sets}\label{A}

Throughout this section we will adopt the notation
$E_\omega( \cdot ) = \int_I \cdot(\omega) d\omega$
for expectation with respect to the randomizing variable. 
Also, we write $a_n \sim b_n$  if there is a constant $C>1$ such 
that $C^{-1}b_n \leq  a_n \leq Cb_n$  for all $n$. 

\newpage

\begin{proposition}\label{prop_expectation_estimates}
For all $0< \alpha < \beta < \infty$ we have
\begin{itemize}
\item[(1)] $$E_\omega(x_n(\omega)) \sim n^{-\frac{1}{\alpha}};$$
\item[(2)]  $$E_\omega(x_n^\prime(\omega)-\frac{1}{2}) \sim
n^{-\frac{1}{\alpha}};$$
\item[(3)] $$m\times m\{ \hat R >n\} \sim n^{1-\frac{1}{\alpha}}.$$
\end{itemize}
\end{proposition}

Before proving this result, we gather some estimates in a sequence of lemmas.

\begin{lemma}\label{lem_domination}
 For all $x \in [0, \frac{1}{2}]$ $T_\alpha(x) \geq T_\beta(x)$ with strict inequality on the open interval $(0,\frac{1}{2})$. 
 \end{lemma}
\begin{proof} This is a straightforward calculation.
\end{proof}
\begin{corollary}\label{cor_domination}
For $0\leq x \leq y < \frac{1}{2}$ we have 
$T_\alpha (y) \geq T_\beta(x)$ with strict inequality in either 
situation:  $0<x\leq y < \frac{1}{2}$ or $0 \leq x < y \leq \frac{1}{2}$.
\end{corollary}
 
We will estimate the position of $x_n(\omega)$ by comparing to the sequence of non-random backwards iterates constructed with only one map; either 
always choosing $T_\alpha|_{[0,\frac{1}{2}]}^{-1}$ or 
$T_\beta|_{[0,\frac{1}{2}]}^{-1}$ in place of $T_\alpha(\omega)|_{[0,\frac{1}{2}]}^{-1}$ in equation (\ref{def_backorbit}).  Denote these 
non-random iterates by $x_n^\alpha$ and $x_n^\beta$ respectively. 
It is immediate from Lemma \ref{lem_domination} that for every $n,~ x_n^\alpha \leq x_n^\beta$.  
Furthermore, 
it is well-known that $x_n^{\alpha}\sim n^{-\frac{1}{\alpha}}$ with 
similar estimates for the parameter $\beta$.
(See, for example, estimates at the beginning of Section 6.2 of \cite{Y}.)

We begin with a very rough (but intuitively obvious) estimate 
on $x_n(\omega)$.

\begin{lemma}\label{lem_rough_estimates}
For all $n \geq 1$ and for all $\omega$ 
$$ x_n^\alpha \leq x_n(\omega)\leq x_n^\beta.$$
\end{lemma}
\begin{proof}
Suppose, to the contrary, $x_n(\omega) < x_n^\alpha$ for some $n,\omega$.  Note that if 
$\alpha(\varphi^k(\omega))= \alpha$ for all $k$ then $x_n(\omega) = x_n^\alpha$, contradicting our assumption.   Let $k\in \{ 0, 1, \dots n-1\}$ be smallest integer such that 
$\alpha(\varphi^k(\omega))= \beta$. 
Then $x_{n-k}(\varphi^k(\omega)) < x_{n-k}^\alpha$ since $T^k_\alpha$ is increasing 
and 
$$x_{n-k-1}(\varphi^{k+1}(\omega)) = T_\beta(x_{n-k}(\varphi^k(\omega)))
< T_\alpha(x_{n-k}(\varphi^k(\omega)) = x_{n-k-1}^\alpha.$$
Here we have invoked Corollary \ref{cor_domination}. 
Iterating this argument for each index where $\alpha(\varphi^j(\omega)) = \beta$ gives 
$$\frac{1}{2} = x_1(\varphi^{n-1}(\omega)) < x_1^\alpha = \frac{1}{2}$$
which is again a contradiction.  
A similar argument shows $x_n(\omega) \leq x_n^\beta$ for all 
$n, \omega$.  
\end{proof}

\begin{lemma}\label{lem_K_0}
Suppose $n$ is given and $K_0\in [0,n-1]$ is fixed. 
 Suppose $\omega \in [0,1]$ is such that
$$\#\{ j \in \{0, 1, \dots n-1\}~ | ~\alpha(\varphi^j(\omega))=\alpha\} > K_0.$$
Then $x_n^\alpha \leq x_n(\omega) \leq x_{\lfloor K_0 \rfloor}^\alpha$
\end{lemma}

\begin{proof}  The left-hand inequality is given by Lemma \ref{lem_rough_estimates}.  For the other side, 
suppose $ x_n(\omega)>x_{\lfloor K_0 \rfloor}^\alpha $.  Consider the following 
sequence of points 
$$
y_{n-j}:= \left\{
\begin{array}{ll}
T_\alpha(x_n(\varphi^j \omega)) & \textnormal {if } \alpha(\varphi^j \omega)= \alpha, \cr
x_n(\varphi^j \omega) & \textnormal{if } \alpha(\varphi^j \omega)= \beta.\end{array}
\right . 
$$
For each $j=0,1, \dots n-1$ define
$K(j) := \#\{ i \in \{0, 1, \dots j \} ~|~ \alpha(\varphi^i \omega) = \alpha\}$.

For example, $K(0)=0$ or  $1$, $K(j) \leq j+1$ and $K(n-1) > \lfloor K_0 \rfloor$ by hypothesis. 

By an argument similar to the proof of Lemma \ref{lem_rough_estimates}, using $T_\beta$ compared to the identity map we have 
$$ y_{n-j}\leq x_{n-j}(\omega) , \textnormal{ for all } 
j= 0, 1, \dots n-1.$$
On the other hand, comparing $T_\alpha$ to the identity map and applying Lemma \ref{lem_domination} gives
$$ x_{\lfloor K_0 \rfloor -j}^\alpha < y_{n-j}, \textnormal{ for all } 
j= 0, 1, \dots n-1.$$

Pick $j_0$ so that $K(j_0) = \lfloor K_0 \rfloor - 1$. Note that $j_0 < n-1$.
Then

$$\frac{1}{2} = x_1^\alpha=x^\alpha_{\lfloor K_0 \rfloor - (\lfloor K_0 \rfloor -1)} < y_{n-j_0} \leq x_{n-j_0}(\omega)$$
which contradicts the hitting time of $x_n(\omega)$ to the interval 
$[\frac{1}{2}, 1]$.
\end{proof}

Pick any $0<p_0 < p_1$, fix $n>1$ and let $K_0 :=n p_0$. 
There are many standard large deviation estimates for i.i.d. random variables that will ensure that \emph{most} $\omega$ encounter at least $K_0$ instances of $\alpha(\varphi^i \omega) = \alpha$ in their first 
$n$ iterates.  As we are aiming for exponential decay in the tail estimate, we invoke a classical result due to Hoeffding \cite{Ho} that works especially well for our case of Bernoulli random variables.  It is precisely at this point that we avoid generating an upper bound constraint on $\beta$ as was the case in Gou\"ezel \cite{G2}.  If instead we were to use the more
general estimates from the well-known Berry-Ess\'een Theorem (e.g. Theorem 1, Section XVI.5 in \cite{F}), for example,  we would obtain power law decay in the tail leading to the requirement $\beta < \frac{3}{2} \alpha$ in order to complete the proof.

\begin{lemma}\label{lem_hoeffding}
For every $n\geq 1$ 
$$Pr\{\omega ~|~ \#\{j \in \{0,1, \dots n-1\} ~|~ \alpha(\varphi^j \omega) = \alpha\}  \leq K_0\} \leq  \exp[ -2n(p_1 - p_0)^2].$$
\end{lemma}
\begin{proof}
Let $S_n$ count the number of times the value $\beta$ occurs in the first $n$ iterates. 
Observe that 
\begin{equation}\label{eqn_hoeffding}
\begin{split}
Pr\{\omega ~|~ \#\{j \in \{0,1, \dots n-1\} ~|~ \alpha(\varphi^j \omega) = \alpha\}  \leq K_0\}\\
\leq
Pr\{\omega ~|~ \#\{j \in \{0,1, \dots n-1\} ~|~ \alpha(\varphi^j \omega) = \beta\}   \geq n-K_0\}
\end{split}
\end{equation}
In Theorem 1 of \cite{Ho} let $\mu=p_2$ and let $t= p_1 - p_0 < p_1= 1- \mu$.
Then the bottom probability in equation (\ref{eqn_hoeffding}) equals
$$
Pr\{ S_n - \mu\,  n \geq (1- p_0 - p_2) \, n\}
= Pr\{\frac{S_n}{n} - \mu  \geq t \}.$$
The exponential estimate now follows from (2.3) in Theorem 1 of \cite{Ho}.
\end{proof}

\newpage

\begin{proof}(Of Proposition \ref{prop_expectation_estimates})
\begin{itemize}
\item[(1)]  For fixed $n$, with $p_0< p_1$ as above, let $K_0=p_0 n$. Set $$G_n = \{ \omega ~|~ \{\omega ~|~ \#\{j \in \{0,1, \dots n-1\} ~|~ \alpha(\varphi^j \omega) = \alpha\}  > K_0\}.$$
Lemma \ref{lem_hoeffding} estimates 
$Pr(I \setminus G_n) \leq \exp[-2n(p_1-p_0)^2]$.

Now 
$$E_\omega(x_n(\omega)) = \int_{G_n} x_n(\omega) \, d\omega
+ \int_{I \setminus G_n} x_n(\omega) \, d\omega
\leq x_{\lfloor K_0 \rfloor}^\alpha  + \frac{1}{2} Pr(I \setminus G_n)$$
where we have used Lemma \ref{lem_K_0} for the first term and
the fact that $x_n(\omega) \leq \frac{1}{2}$ for the second term. 
Now
$x_{\lfloor K_0 \rfloor}^\alpha \leq C_1(K_0^{-\frac{1}{\alpha}})
\leq C_2(n^{-\frac{1}{\alpha}})$ when $K_0 = p_0 n$.
On the other hand, the second term tends to zero exponentially fast. Since
$x_n(\omega) \geq  x_n^\alpha \geq C_3 n^{-\frac{1}{\alpha}}$ by 
Lemma \ref{lem_rough_estimates} and the fact that $x_n^\alpha \sim n^{-\frac{1}{\alpha}}$ we 
have established the required estimate on the expectation. 
\item[(2)]
This follows from part (1) immediately, since both maps have the same linear second branch.  
\item[(3)]  
We have to get an estimate on
$$(m\times m)\{\hat R >n\}=\sum\limits_{l\geq n+1}(m\times m)(\Delta_l)=\sum\limits_{l\geq n+1}\sum\limits_{i\geq l+1}(m\times m)(\Delta_{0,i}).$$
First, observe that
\begin{equation}\label{boxes}
\begin{split}
(m\times m)(\Delta_{0,i})
&=\sum\limits_{j=1}\limits^{2^i}(m\times m)(\Delta^j_{0,i})=\int\limits_0\limits^1 J_i(\omega)d\omega\\
&=\int\limits_0\limits^1 [x'_{i-1}(\omega)-x'_{i}(\omega)]d\omega\\
&=E[x'_{i-1}(\omega)]-E[x'_{i}(\omega)]\\
&=\frac{1}{2}[E(x_{i-1}(\omega))-E(x_{i}(\omega))].
\end{split}
\end{equation}
Therefore, by equation (\ref{boxes}) and part (1) of this proposition we have
\begin{equation*}
\begin{split}
(m\times m)\{\hat R >n\}
&=\frac{1}{2}\sum\limits_{l\geq n+1}\sum\limits_{i\geq l+1}[E(x_{i-1}(\omega))-E(x_{i}(\omega))]\\
&=\frac{1}{2}\cdot[(Ex_{n+1}(\omega)-Ex_{n+2}(\omega))+2(Ex_{n+2}(\omega)-Ex_{n+3}(\omega))\\
&+3(Ex_{n+3}(\omega)-Ex_{n+4}(\omega))+...]\\
&=\frac{1}{2}\cdot[E(x_{n+1}(\omega))+E(x_{n+2}(\omega))+E(x_{n+3}(\omega))+...]\\
&\sim\sum\limits_{k>n} k^{-\frac{1}{\alpha}}\\
&\sim  n^{1-\frac{1}{\alpha}}.
\end{split}
\end{equation*}
\end{itemize}
\end{proof}

\subsection{Distortion}\label{B}
\begin{lemma}\label{distcoeff}
If $(x,\cdot),(y,\cdot)\in \Delta_0$ and $s((x,\cdot),(y,\cdot))=n,$ then
$|x-y|\leq \theta ^n.$
\end{lemma}
\begin{proof}
Set $\theta:=\frac{1}{2}<1$ and observe that on $\Delta_{0}, DT_\omega^R\geq 2=\theta^{-1}.$ Thus, if 
$(x,\cdot),(y,\cdot)$ lie in a common atom $\Delta^j_{0,i}$ such that $x,y \in J_i(\omega)\subseteq (T_\omega^R)^{-1}(\frac{1}{2},1],$
then
$$\min DT_\omega^R\leq \left|\frac{T_\omega^Rx-T_\omega^Ry}{x-y}\right|\leq \frac{1}{|x-y|}.$$
Therefore, $|x-y|\leq \theta$ and the result follows by induction on $k\leq n.$
\end{proof}

\begin{lemma}\label{distlog}
There exists a constant $C>0$ such that for $(x,\cdot),(y,\cdot)\in \Delta^j_{0,i},$
$$|\log\frac{DT^R_\omega(x)}{DT^R_\omega(y)}|\leq C|T^R_\omega(x)-T^R_\omega(y)|\leq C.$$
\end{lemma}

\begin{proof}
It is trivial for $J_i,i=1$ since $T_{\alpha(\omega)}(x)=2x-1.$
We apply the Koebe principle to prove the result for $J_i,i\geq 2.$

Recall the Schwarzian derivative of a function $f \in C^3$ is given by:
$$({\bf S} f)(x)=\frac{f'''(x)}{f'(x)}-\frac{3}{2}(\frac{f''(x)}{f'(x)})^2.$$
It is also well known that the Schwarzian derivative of the composition of two functions $h, f\in C^3$ satisfies
$${\bf S}(h\circ f)(x)=({\bf S}h)(f(x))\times (f'(x))^2+({\bf S}f)(x).$$
Consequently, Schwarzian derivative of the composition is negative if both functions have negative Schwarzian derivatives. Let $g$ denote the composition of the left branches of $T_{\alpha(\varphi^{R-1}\omega)},...,T_{\alpha(\varphi\omega)}$ and the right branch
of $T_{\alpha(\omega)}$. Notice that on $J_i(\omega), i\geq 2,$ 
we have $g(x)= T_\omega^R(x).$ Since $0<\alpha<\beta\leq 1$, we have for the left branch $T'_{\alpha(\omega)}> 0$, $T''_{\alpha(\omega)}>0$ and $T'''_{\alpha(\omega)}\le 0$; in particular, $T'''_{\alpha(\omega)}= 0$ if and only if $\alpha(\omega)=\beta=1$. Thus, ${\bf S}g< 0.$ 

\bigskip

For each $J_i(\omega),i\geq 2,$ let  $J=[x'_{i+1}(\omega), 2]$. Note that $g(x'_{i+1}(\omega))<\frac12$. Set $\kappa:=\frac12-\sup_{\omega}g(x'_{i+1}(\omega))>0$. Then 
 $J_i(\omega)\subset J$ and
$g(J_i(\omega))=(\frac{1}{2},1]\subset (\frac12-\kappa,2]\subset g(J).$ This means $g(J)$ contains a $\kappa$-scaled neighborhood of $g(J_i(\omega))$ with constant $\kappa.$ Therefore, by Koebe principle \cite{MV} there exists a constant $C(\kappa)>0$ such that
$$|\log\frac{g'(x)}{g'(y)}|\leq C(\kappa)\frac{|x-y|}{|J_i(\omega)|}\leq C(\kappa),\quad \forall x,y \in J_i(\omega),$$
and consequently, 
$$|\frac{g'(x)}{g'(y)}|\leq e^{C(\kappa)}.$$ 
It follows that
$$\frac{|x-y|}{|J_i(\omega)|}\leq e^{C(\kappa)}\cdot\frac{|g(x)-g(y)|}{|g(J_i(\omega))|}.$$
Hence, $|\log\frac{g'(x)}{g'(y)}|\leq C(\kappa)\cdot e^{C(\kappa)}\cdot\frac{|g(x)-g(y)|}{|g(J_i(\omega))|}$, which completes the proof.
\end{proof}
\begin{remark}\label{positive_schwarzian}
Note that for $\beta>1$, $({\bf S} T_{\beta}(x))>0 $ for $x\in \left(0,\frac12\left(\frac{\beta-1}{(1+\beta)(1+\frac{\beta}{2})}\right)^{1/\beta}\right)$.
\end{remark}
\begin{proposition}\label{distortion}
There exists a constant $C(F)>0$ such that for $z_1,z_2\in \Delta^j_{0,i},$
$$|\frac{DF^R(z_1)}{DF^R(z_2)}-1|\leq C(F)\cdot\theta ^{s(F^R(z_1),F^R(z_2))}.$$
\end{proposition}
\begin{proof}
Let $z_1=(x_1,\omega_1),z_2=(x_2,\omega_2)\in \Delta^j_{0,i}.$ Then they have same realization
$$(\alpha(\omega_l),\alpha(\varphi\omega_l),...,\alpha(\varphi^{R-1}\omega_l)),$$ for $j=1,2.$ Using this fact and
$F^R(z_l)=S^R(z_l),$ for $z_l \in \Delta_{0,i}, l=1,2$, we have:
\begin{equation*}
\begin{split}
\frac{DF^R(z_1)}{DF^R(z_2)}&=\frac{DS^R(x_1,\omega_1)}{DS^R(x_2,\omega_2)}\\
&=\frac{\left|\begin{array}{cc}
                DT^R_{\omega_1}(x_1) & \frac{\partial T^R_{\omega}}{\partial\omega}(z_1) \\
                0 & D\varphi^R(\omega_1) 
              \end{array}
\right|}{\left|\begin{array}{cc}
                DT^R_{\omega_2}(x_2) & \frac{\partial T^R_{\omega}}{\partial\omega}(z_2) \\
                0 & D\varphi^R(\omega_2)
              \end{array}
\right|}\\
&=\frac{DT^R_{\omega_1}(x_1)\cdot\frac{1}{P^R_{\omega_1}}}{DT^R_{\omega_2}(x_2)\cdot\frac{1}{P^R_{\omega_2}}}\\
&=\frac{DT^R_{\omega_1}(x_1)}{DT^R_{\omega_2}(x_2)}=\frac{DT^R_{\omega}(x_1)}{DT^R_{\omega}(x_2)},
\end{split}
\end{equation*}
for any $\omega\in \Delta^j_{0,i}.$ By using Lemma \ref{distcoeff}, Lemma \ref{distlog} and the following inequality:
$$|x-1|\leq\frac{e^C-1}{C}|\log x|, \text{ if } |\log x|\leq C,$$
we obtain
\begin{equation*}
\begin{split}
|\frac{DT^R_\omega(x_1)}{DT^R_\omega(x_2)}-1|
&\leq\frac{e^C-1}{C}|\log \frac{DT^R_\omega(x_1)}{DT^R_\omega(x_2)}|\\
&\leq\frac{e^C-1}{C}\times C|T^R_\omega(x)-T^R_\omega(y)|\\
&\leq C(F)\cdot\theta ^{s(F^R(z_1),F^R(z_2))}.
\end{split}
\end{equation*}
\end{proof}
\bibliographystyle{amsplain}

\end{document}